\documentclass{amsart}

\usepackage{amsmath,amssymb,url}
\usepackage{enumitem}

\numberwithin{equation}{section}

\theoremstyle{plain}
\newtheorem{theorem}{Theorem}[section]
\newtheorem{lemma}[theorem]{Lemma}
\newtheorem{proposition}[theorem]{Proposition}
\newtheorem{corollary}[theorem]{Corollary}

\theoremstyle{definition}
\newtheorem{definition}[theorem]{Definition}
\newtheorem{notation}[theorem]{Notation}

\newtheorem{construction}[theorem]{Construction}
\newtheorem{claim}{Claim}[theorem]
\newtheorem{subclaim}{Subclaim}
\newtheorem{problem}[theorem]{Open Problem}

\newcommand{\alg}[1]{\mathbf{#1}}
\newcommand{\var}[1]{\mathcal{#1}}
\newcommand{\Th}[1]{\mathrm{Th}\left( #1 \right)}
\newcommand{\fin}{\mathrm{fin}}
\newcommand{\Con}[1]{\mathrm{Con}\left(#1 \right)}
\newcommand{\HSP}[1]{\mathrm{HSP}\left(#1 \right)}
\newcommand{\inangle}[1]{\left \langle #1 \right \rangle}
\newcommand{\vpair}[2]{\begin{pmatrix} #1 \\ #2 \end{pmatrix}}
\newcommand{\Rad}[1]{\mathrm{Rad}\left(#1 \right)}
\newcommand{\frz}[2]{{#1}^{#2}}
\newcommand{\frzflt}[2]{{#1}^{#2 \flat}}

\newcommand{\Pol}[2]{\mathrm{Pol}_{#1}\left( #2 \right)}
\newcommand{\Cg}[3]{\mathrm{Cg}_{#1}\left\langle #2, #3 \right\rangle}
\newcommand{\TC}[3]{\mathrm{C}\left( #1, #2 ; #3 \right)}
\newcommand{\point}[1]{\mathbf{#1}}
\newcommand{\NRINV}{\mathsf{NRINV}}
\newcommand{\GEN}{\mathsf{GEN}}
\newcommand{\EDGEGEN}{\mathsf{EDGEGEN}}
\newcommand{\VERTEXGEN}{\mathsf{VERTEXGEN}}
\newcommand{\EDGE}{\mathsf{EDGE}}
\newcommand{\VERTEXRED}{\mathsf{VERTEXRED}}
\newcommand{\VERTEXBLUE}{\mathsf{VERTEXBLUE}}

\newcommand{\pfend}[1]{\hspace{\stretch{1}}\(\dashv_{#1}\)}

\begin{document}

\title[Arity bounds in FD varieties]{Bounding essential arities of term operations in finitely decidable varieties}

\author[M. Smedberg]{Matthew Smedberg}
\email{matthew.smedberg@vanderbilt.edu}
\urladdr{http://www.vanderbilt.edu/math/people/smedberg}
\address{Mathematics Department\\
Vanderbilt University\\Nashville, TN 37240\\USA}

\begin{abstract}
  Let \( \alg{A} \) be a finite algebra generating a finitely decidable variety and having nontrivial strongly solvable radical \( \tau \). We provide an improved bound on the number of variables in which a term can be sensitive to changes within \(\tau\). We utilize a multi-sorted algebraic construction, amalgamating the methods developed by Valeriote and McKenzie for the investigation of strongly abelian locally finite decidable varieties with those of Idziak for locally finite congruence modular finitely decidable varieties.
\end{abstract}

\maketitle

Among the algorithmic properties most investigated by algebraists is the problem of when a given computably axiomatizable class \( \var{K} \) of first-order structures will have computable first-order theory too. This problem was investigated for varieties of groups and rings beginning in the 1950s, with signal contributions from Tarski and his students in the USA (\cite{TMR1953}, \cite{Sz1955}) and from the Russian school of Luzin, Ershov et al. (\cite{Malc1960}, \cite{Ershov1972}, \cite{Zam1976}, \cite{Zam1978}).

For many but not all interesting classes \( \var{K} \), it was shown that not only is \( \Th{\var{K}} \) undecidable, but \( \Th{\var{K}_{\mathrm{fin}}} \) may be as well, where \( \var{K}_{\mathrm{fin}} \) denotes the class of all finite structures in \( \var{K} \). We will say that \( \var{K} \) is \emph{(finitely) decidable} if \( \Th{\var{K}} \) (resp.~\( \Th{\var{K}_{\mathrm{fin}}} \)) is a computable set of sentences.

For example, any variety of groups has decidable theory iff it contains only abelian groups (as is showed in \cite{Sz1955} and \cite{Zam1978}). Szmielew actually showed more: \emph{every} theory of abelian groups is decidable, which together with the famous construction by Olshanskii of a variety of groups whose smallest nonabelian member is infinite (\cite{Olsh1991}), shows that a variety can be undecidable and simultaneously finitely decidable. (Zamyatin had given an earlier example of this for varieties of rings in \cite{Zam1976}.)

We restrict our attention in this paper to varieties of abstract algebras in a finite language. The natural questions here are: given a computably axiomatizable variety \( \var{V} \) (in particular, a variety of the form \( \HSP{\alg{A}} \) for some finite algebra \( \alg{A} \)), is \( \Th{\var{V}} \) (resp. \( \Th{\var{V}_{\mathrm{fin}}} \)) a computable set of sentences? One immediately sees that \( \Th{\var{V}} \) is computably enumerable, so the one question is equivalent to the computable enumerability of the set of sentences refutable in some member of \( \var{V} \); on the other hand, it is also clear that the set of sentences \emph{refuted} in some finite member of \( \var{V} \) is computably enumerable, while the set of sentences true in all these algebras may not be.

In \cite{MV1989}, McKenzie and Valeriote showed that locally finite \emph{decidable} varieties have a very restricted structure theory. Such a variety must decompose as the varietal product of a discriminator variety, a variety of modules, and a strongly abelian variety. In particular, 

\begin{corollary}\label{cor:SS->SA}
If \( \var{V} \) is a locally finite decidable variety, then every strongly solvable congruence of an algebra in \( \var{V} \) is strongly abelian.\end{corollary}

While the analogues betwen the decidability problem and the finite decidability problem are strong, not all the necessary conditions for decidability transfer down; Corollary \ref{cor:SS->SA} does, however (\cite{McSmed2013}) and we will make implicit use of it throughout this paper.

One of the properties that does \emph{not} continue to hold is the direct decomposition theorem. In \cite{Idz1997}, P.~Idziak gave a characterization of finitely decidable locally finite varieties with modular congruence lattices; this characterization essentially gives a recipe for building a variety with no possible direct decomposition into discriminator and affine varietal factors. One goal of the present paper is to suggest a potential reformulation of the direct product criterion to make it work in the finitely decidable setting.

As mentioned, Idziak's characterization extends only to congruence-modular varieties. By \cite{VW1992} in conjunction with \cite[Theorem 8.5]{HM1988}, a locally finite finitely decidable variety \( \var{V} \) is congruence-modular iff no algebra in \( \var{V} \) has a nontrivial strongly solvable congruence. (For a quick grounding in the notions of (strong) abelianness and solvability, see \cite[Chapter 0]{MV1989}. As mentioned, we will assume that the notions of ``strongly solvable'' and ``strongly abelian'' coincide in all the varieties considered here.) In our hopes to extend Idziak's characterization, we will be guided by a construction in \cite[Chapter 11]{MV1989}, which takes a strongly abelian first-order variety \( \var{V} \) and recasts it as a variety \( \var{V}^\flat \) in a multi-sorted language; the main theorem associated with this construction asserts that \( \var{V} \) is decidable iff it is finitely decidable, and both are equivalent to \( \var{V}^\flat \)'s being essentially unary.

The plan of this paper is as follows: We will very quickly state definitions and needed theorems from the literature in Section \ref{sec:review}. Then, since the construction of \( \var{V}^\flat \) does not carry over directly to a nonabelian setting, we build the appropriate generalization, constructing from a fixed finite algebra \( \alg{A} \) with a strongly abelian congruence \( \tau \) the multi-sorted first-order languages \( \frz{L}{\tau} \) in Section \ref{sec:frz} and \( \frzflt{L}{\tau} \) in Section \ref{sec:frzflt}. Finally, Section \ref{sec:main interpretation} will contain the proof of the main theorem, which proceeds by semantic interpretation.

\section{Definitions and Preliminaries}\label{sec:review}

\begin{definition}
  Let
  \[ X_1 \times X_2 \times \cdots \times X_n \stackrel{f}{\rightarrow} Y \]be a function. We say that \(f\) \emph{depends essentially} on its \(i^{\text{th}}\) variable if there exist \(a \neq a' \in X_i \) and \(b_j \in X_j \; (j \neq i)\) so that
  \[ f(b_1, b_2, \ldots, b_{i-1},a,b_{i+1},\ldots,b_n) \neq f(b_1, b_2, \ldots, b_{i-1},a',b_{i+1},\ldots,b_n) \](Clearly, if \(f\) depends on its \(i^{\text{th}}\) variable, it follows that \(|X_i| > 1\).)
\end{definition}

In particular, if \( f \) is a term of the (ordinary first-order) algebra \( \alg{A} \), unless otherwise specified each \( X_i \) is \( A \); if \( \alg{M} \) is a multi-sorted algebra, the default assumption is that each \( X_i \) is the entire sort associated to the corresponding input variable of \( f \).

\begin{definition}
  Let \(A\) be a finite set. We say that the operation \(d(v_1, \ldots, v_K)\) is a \emph{decomposition operation} on \( A \) if \begin{itemize}
    \item \(d(A,\ldots, A) \subseteq A \);
    \item the action of \(d\) on \(A\) depends on all its variables;
    \item \(d(x,\ldots, x) = x \) for all \( x \in A \); and
    \item \begin{align}\label{eq:decomp} & d(d(x_{1,1},\ldots, x_{1,K}),d(x_{2,1},\ldots,x_{2,K}),\ldots,d(x_{K,1},\ldots,x_{K,K})) \notag \\ &\qquad = \\ &d(x_{1,1},x_{2,2},\ldots,x_{K,K})\notag \end{align} for all \(x_{i,j} \in A\).
  \end{itemize}

  Typically, we will have in mind an algebraic structure on \(A\) or perhaps on some superset of \(A\). If the operation \(d\) is a term operation (resp.~polynomial operation) of the structure \( \alg{A} \), we will call it a \emph{decomposition term} (resp.~\emph{decomposition polynomial}).
\end{definition}

\begin{proposition}[{\cite[Lemma 11.3]{MV1989}}]
  If \( \alg{A} \) is a strongly abelian algebra having an idempotent term \(t(v_1, \ldots, v_K)\) depending essentially on all its variables, then \( \alg{A} \) has a decomposition term of arity \(K\).
\end{proposition}

It follows that in such an algebra, if \(t\) is a term which depends on all its variables and such that \(t(x,x,\ldots,x) \) is a permutation, then there is a decomposition term of the same arity as \(t\). 

Decomposition operators have a nice description in the case where \( \alg{A} \) is strongly abelian:

\begin{proposition}[{\cite[Lemma 11.4]{MV1989}}]
  If \( \alg{A} \) is a finite strongly abelian algebra and \(K\) the largest arity of a decomposition term \(d\) on \( \alg{A} \), then there exist finite sets \(A_1, \ldots, A_K \) and an isomorphism \(\varphi\) from \( \alg{A} \) to a structure \( \alg{B} \) with underlying set \( A_1 \times \cdots \times A_K \) such that, if we denote
\[ \varphi(a) = \begin{matrix} a^1 \\ \vdots \\ a^K \end{matrix} \]
then
\[ d^{\alg{B}}(\varphi(a_1), \varphi(a_2), \ldots, \varphi(a_K)) = \varphi \left(d^{\alg{A}} \begin{pmatrix}
  a_1^1 & a_2^1 & \cdots & a_K^1 \\
  a_1^2 & a_2^2 & \cdots & a_K^1 \\
  \vdots & \vdots & \ddots & \vdots \\
  a_1^K & a_2^K & \cdots & a_K^K
\end{pmatrix} \right)=
\begin{matrix}
  a_1^1 \\ a_2^2 \\ \vdots \\ a_K^K
\end{matrix} \]
\end{proposition}

In \cite[Theorem 11.9]{MV1989}, McKenzie and Valeriote showed that

\begin{theorem}\label{thm:MV11.9}
  If \( \alg{A} \) is strongly abelian and \(K\) the largest arity of a decomposition term over \( \alg{A} \), then any other term's depending on more than \(K\) variables implies that \( \Th{\var{V}} \) and \( \Th{\var{V}_\fin} \) are undecidable for any variety \( \var{V} \) containing \( \alg{A} \).
\end{theorem}

Our goal is to generalize this result to algebras \( \alg{A} \) which are not themselves strongly abelian, but do contain nontrivial strongly abelian congruences. 

\begin{proposition}\label{prop:SA induced algebra}
  Let \( \alg{A} \) be a finite algebra with a strongly abelian congruence \( \tau \). Let \(C \subset A \) be any \( \tau \)-class; then the non-indexed algebra
  \[ \alg{A}_{|C} = \langle C; \; \{f_{|C} \colon f \in \Pol{}{\alg{A}},\, f(C,C,\dots,C) \subseteq C \} \rangle \]is strongly abelian.
\end{proposition}

It would be natural to search for a generalization of Theorem \ref{thm:MV11.9} by looking at polynomials which restrict to decomposition operations on \( \tau \)-classes; however, we found this approach to have attendant difficulties.

Instead, we make the following definition:

\begin{definition}Let \( \alg{A} \) be a finite algebra with a congruence \( \tau \) as above. Suppose we have a term \(t(v_1, \ldots, v_n, v_{n+1}, \ldots, v_{n+k}) \) of \( \alg{A} \) and \(\tau\)-classes such that the action
  \[ C_1 \times \cdots \times C_n \times C_{n+1} \times \cdots \times C_{n+k} \stackrel{t}{\rightarrow} C_0 \]does not depend on the variables \( n+1 \) through \( n + k \). We call the map
  \begin{align*} f \colon C_1 \times \cdots C_n &\rightarrow C_0 \\
    \vec{x} &\mapsto t(\vec{x}, \vec{a}) \end{align*}(\( \vec{a} \) any arbitrary tuple from \( C_{n+1} \times \cdots \times C_{n+k} \)) a \emph{\(\tau\)-boxmap}.
\end{definition}

The remainder of the paper is devoted to proving the following theorem:

\begin{theorem}\label{thm:main}
  Let \( \tau \) be the strongly solvable radical of the finite algebra \( \alg{A} \). Fix any term \(t(v_1, \ldots, v_n) \) and let \( C_0, C_1, \ldots, C_n \) be \(\tau\)-classes such that 
  \[ C_1 \times \cdots \times C_n \stackrel{t}{\rightarrow} C_0 \]Let \(K\) be the maximum arity of a decomposition \(\tau\)-boxmap on \(C_0\).

  Then if the action of \(t\) on \(C_1 \times \cdots \times C_n \) depends on more than \(K\) factors, it follows that \( \HSP{\alg{A}} \) is hereditarily finitely undecidable.
\end{theorem}

In Section \ref{sec:main interpretation}, we will need the following definitions:

\begin{definition}Let \( \alg{A} \) be any algebra.
  \begin{enumerate}
  \item We say that a term \( t(v_1, \ldots, v_n) \) is \emph{left-invertible} at \(v_i\) if there exists a term \(r(v_0, v_{n+1}, \ldots, v_{n+k}) \) such that
    \[ \alg{A} \models v_i = r(t(v_1, \ldots, v_n), v_{n+1},\ldots,v_k) \]
  \item Likewise we call \(t(v_1, \ldots, v_n) \)  \emph{right-invertible} if there exist terms \newline\(s_i(v_0, \ldots, v_\ell)\), \( 1 \leq i \leq n\), such that
    \[ \alg{A} \models t(s_1(v_0,\ldots,v_\ell), \ldots, s_n(v_0,\ldots,v_\ell)) = v_0 \]
  \end{enumerate}
\end{definition}

\begin{notation}
  Let \( \alg{A} \) be a structure and \(I\) a (large) index set. We will use a bold \( \point{x} \) to denote a member of \( \alg{A}^I \), and call such elements ``points''. The value of \( \point{x} \) at the \(i^{\text{th}}\) coordinate will be \( x^i \), and we will write
  \[ \point{x} = x_{|I_0} \oplus y_{|I_1} \oplus \cdots \]to express that \(I_0 \) is the subset of coordinates \(i\) where \(x^i = x \), \(I_1\) the subset where \( x^i = y \), etc. If \( I_0 \) is a singleton we will write
  \[ \point{x} = x_{|i} \oplus \cdots \]instead of
  \[ \point{x} = x_{|\{i\}} \oplus \cdots. \]
\end{notation}

In this paper, the proof of a theorem, lemma, etc.~will be marked as usual with\hspace{\stretch{1}}\(\qed\) 

The proof of a claim within a larger proof will be marked with a turnstile indicating the claim number, as follows:\pfend{0.1}

\section{The language \(\frz{L}{\tau}\)}\label{sec:frz}

We will be building two multi-sorted languages from which to effect an interpretation. While it is possible to formalize multi-sorted model theory entirely in a usual first-order setting, this formalization takes away much of the naturality of the multi-sorted definition. In particular, the first-order formalization ``gets wrong'' the structural operations of direct product and substructure; these are key for us, since we will be constructing varieties in our sorted model classes.

\begin{notation}
  Every atomic formula \( \Phi(v_1, v_2, \ldots) \) of a multi-sorted language must implicitly or explicitly determine what sort each variable must be assigned from. We call this the \emph{type signature} of the formula. In particular, for a term \(t\) we write
  \[ t(S_1, S_2,\ldots) \rightarrow S_0 \]to denote that the formula
  \[ t(x_1, x_2, \ldots) = x_0 \]is meaningful only if \(x_0 \in S_0\), \(x_1 \in S_1 \), \(x_2 \in S_2 \), and so forth.
\end{notation}

For the remainder of this and the next section, fix a finite algebraic language \(L\) and a finite \(L\)-algebra \( \alg{A} \) with a congruence \( \tau \) whose congruence classes are \( C_1, \ldots, C_M \).

\begin{definition}
  The multi-sorted first-order language \( \frz{L}{\tau} \) will have the following nonlogical symbols:

  For each \(1 \leq i \leq M \), the language will have a sort symbol \( \inangle{i} \).

  For each basic operation symbol \(f(v_1, \ldots, v_n) \) of \(L\) and each \( 1 \leq i_1, \ldots, i_n \leq M \), \( \frz{L}{\tau} \) will have a basic operations symbol \(f_{i_1 \cdots i_n} \) of type signature
  \[ f_{i_1 \cdots i_n} \left(\inangle{i_1}, \inangle{i_2}, \ldots, \inangle{i_n} \right) \rightarrow \inangle{i_0} \]where
  \[ C_{i_1} \times \cdots \times C_{i_n} \stackrel{t^{\alg{A}}}{\rightarrow} C_{i_0}. \]
\end{definition}

\begin{construction}
  \begin{enumerate}[ref=\thetheorem(\arabic*)]
  \item\label{conststep:Atau} We define an \( \frz{L}{\tau} \)-structure \( \frz{\alg{A}}{\tau} \) in the natural way: each sort
    \[ \inangle{i}^{\frz{\alg{A}}{\tau}} = C_i \]and if \(x_k \in C_{i_k}\) for \(1 \leq k \leq n\),
    \[ f^{\frz{\alg{A}}{\tau}}_{i_1 \cdots i_n}(x_1, \ldots, x_n) = f^{\alg{A}}(x_1, \ldots, x_n). \]

  \item\label{conststep:Btau}
    More generally, let \( \alg{B} \) be any \(L\)-structure such with a congruence \( \tau^{\alg{B}} \) such that there exists an isomorphism \( \varphi \colon \alg{A}/\tau \rightarrow \alg{B}/ {\tau^{\alg{B}}} \). Define an \( \frz{L}{\tau} \)-structure \( \frz{\alg{B}}{\tau} \) by declaring
    \[ \inangle{i}^{\frz{\alg{B}}{\tau}} = \varphi(C_i) \]
    and defining the basic operations
    \[ f^{\frz{\alg{B}}{\tau}}_{i_1 \cdots i_n}(x_1, \ldots, x_n) = f^{\alg{B}}(x_1, \ldots, x_n) \]for any \(x_k \in \varphi(C_{i_k})\). Note that the isomorphism \( \varphi \) will usually be clear in practice, so we do not include it as a visible parameter in the symbol \( \frz{\alg{B}}{\tau} \). Similarly, we will usually refer to the distinguished congruence of \( \alg{B} \) as \(\tau\) rather than \( \tau^{\alg{B}} \).
  \end{enumerate}
\end{construction}

The following proposition connecting the structural operations in \(\frz{L}{\tau}\) with those in \(L\) is easy to prove:

\begin{proposition}\label{prop:frzHSP}
  Let \( \alg{M} = \frz{\alg{B}}{\tau} \) and \( \alg{N} = \frz{\alg{C}}{\tau} \). \begin{enumerate}
  \item Let \( \alg{D} \leq \alg{B} \) have nonempty intersection with each \( \tau \)-class; then \( \alg{D} \) satisfies the hypotheses of Construction \ref{conststep:Btau}, and \( \frz{\alg{D}}{\tau} \) is a substructure of \( \alg{M} \). Moreover, every substructure of \( \alg{M} \) is obtained in this way.

  \item Let \( \theta \leq \tau \) be a congruence on \( \alg{B} \); then \( \alg{B}/\theta \) satisfies the hypotheses of Construction \ref{conststep:Btau}, and \( \frz{\left( \alg{B}/\theta \right)}{\tau} \) is a homomorphic image of \( \alg{M} \). Moreover, every homomorphic image of \( \alg{M} \) is obtained in this way.

  \item Let \( \alg{D} \leq \alg{B} \times \alg{C} \) be the subalgebra consisting of all pairs \( \vpair{b}{c} \) such that \( \varphi^{-1}(b/\tau) = \varphi^{-1}(c/\tau) \). Then \( \alg{D} \) satisfies the hypotheses of Construction \ref{conststep:Btau}, and \( \frz{\alg{D}}{\tau} \) is the product of \( \alg{M} \) and \( \alg{N} \) in the sense of \( \frz{L}{\tau} \). (This generalizes to any number of factors.)
  \end{enumerate}
\end{proposition}

The classical proof that a class is equationally axiomatizable iff it is closed under taking products, substructures, and homomorphic images is valid for multi-sorted algebras, so it makes sense to talk about the variety \( \var{V}(\frz{\alg{A}}{\tau}) = \HSP{\frz{\alg{A}}{\tau}} \). A representation of the free algebras in this variety as subalgebras of a direct power of \( \frz{\alg{A}}{\tau} \), where the index set is itself a power of \( \frz{\alg{A}}{\tau} \), does exist; but is not straightforward to write down, and one is better off thinking of free algebras as algebras of terms. Note that the trivial algebra in this variety is the one where each sort is a singleton, i.e. \( \frz{\left(\alg{A}/\tau\right)}{\tau} \).

\begin{lemma}
  \begin{enumerate}[ref=(\arabic*)]
  \item\label{lemitem:tauabelian} The sorted structure \( \frz{\alg{B}}{\tau} \) is abelian (resp. strongly abelian) if and only if the congruence \( \tau \) was a (strongly) abelian congruence of \( \alg{B} \).
  \item\label{lemitem:Vabelian} If \( \alg{A} \) belongs to a finitely decidable variety and \( \tau \) is a (strongly) solvable congruence, then \( \HSP{\frz{\alg{A}}{\tau}} \) is a (strongly) abelian variety.
  \end{enumerate}
\end{lemma}

\begin{proof}
  \ref{lemitem:tauabelian}: A failure of the (strong) term condition \( \TC{\tau}{\tau}{\bot} \) in \( \alg{B} \) is readily convertible into a failure of the corresponding condition \( \TC{\top}{\top}{\bot} \) in \( \frz{\alg{B}}{\tau} \), and vice versa.

  \ref{lemitem:Vabelian}: Recall our assumption that in \( \HSP{\alg{A}} \), strongly solvable congruences are strongly abelian.

  If \( \HSP{\frz{\alg{A}}{\tau}} \) were to fail to be (strongly) abelian, this failure would be witnessed in a finitely generated, and hence finite, structure \( \alg{M} \). We may suppose \( \alg{M} = \alg{N}/\vartheta \), where \( \alg{N} \) is a substructure of a direct power \( \left(\frz{\alg{A}}{\tau} \right)^X \).

  As we saw in Lemma \ref{prop:frzHSP}, this direct power is the image under \( \frz{\bullet}{\tau} \) of the subalgebra \( \alg{P} \) of \( \alg{A}^X \) consisting of all \( \tau \)-constant tuples. Since any failure of (strong) abelianness would project to a failure at some coordinate, \[ \tau^{\alg{P}} = \tau^X \cap (P \times P) \]is (strongly) abelian. Hence \( \left(\frz{\alg{A}}{\tau} \right)^X \) is (strongly) abelian.

  We know that \( \alg{N} = \frz{\alg{B}}{\tau} \) for some \( \alg{B} \leq \alg{P} \), and moreover that \[ \tau^{\alg{B}} = \tau^{\alg{P}} \cap (B \times B); \] it follows any failure of (strong) abelianness in \( \alg{B} \) would have represented one in \( \alg{P} \) already. Hence \( \alg{N} \) is (strongly) abelian.

  Finally, we have that there must exist \( \theta \in \Con{\alg{B}} \) such that \( \frz{ \left( \alg{B}/\theta \right)}{\tau} = \alg{N}/\vartheta = \alg{M} \). But since \( \tau \) is (strongly) abelian in \( \alg{B} \), \( \theta \) is (strongly) solvable, and hence (strongly) abelian as well; and just as in \ref{lemitem:tauabelian} any witness to the failure of the (strong) term condition \( \TC{\top}{\top}{\vartheta} \) in \( \alg{N} \) would give rise to a failure of the corresponding condition \( \TC{\tau}{\tau}{\theta} \) in \( \alg{B} \).
\end{proof}

\begin{corollary}
  If \( \alg{A} \) belongs to any finitely decidable variety and \( \tau \) is either the solvable radical or the strongly solvable radical of \( \alg{A} \), then \( \HSP{\frz{\alg{A}}{\tau}} \) semantically interprets into \( \HSP{\alg{A}} \).
\end{corollary}

\begin{proof}
  The key observation is that each of the congruences in the statement of the theorem is uniformly definable in \( \HSP{\alg{A}} \) (this is proved in \cite{McSmed2013}), and our construction guarantees that \( \tau^{\alg{B}} \) is the (strongly) solvable radical of \( \alg{B} \) whenever \( \tau \) was of \( \alg{A} \).

  Let \( c_1, \ldots, c_M \) be new constant symbols. Take any \( \alg{M} = \frz{\alg{B}}{\tau} \in \HSP{\frz{\alg{A}}{\tau}} \), where \( \alg{M} \) and \( \alg{B} \) can be taken to be on the same underlying set. First, assign \(c_i\) to an arbitrary element of \( \varphi(C_i)\) for each \(i\). Then one can recover the sort of \(x\) by asserting that \(x \) and \( c_i\) are congruent modulo the radical; likewise the assertion \( f_{i_1 \cdots i_n} (x_1, \ldots, x_n) = x_0 \) is true in \( \alg{M} \) iff each \( x_k \equiv_{\Rad{\alg{B}}} c_{i_k} \) and \(f(x_1, \ldots, x_n) = x_0 \) in \( \alg{B} \).

  It follows that whenever \( \HSP{\frz{\alg{A}}{\tau}} \) is (finitely) undecidable and finitely axiomatizable (which happens, for instance, when the variety is strongly abelian), then \( \HSP{\alg{A}} \) is (finitely) undecidable too.
\end{proof}

\section{The language \( \frzflt{\alg{A}}{\tau} \)}\label{sec:frzflt}

The construction in the previous section required no assumptions about \( \tau \). If, however, \( \tau \) is strongly abelian, then we can introduce a further sorted construction, generalizing that effected by McKenzie and Valeriote in \cite[Chapter 11]{MV1989}. For the remainder of this section, we add the assumption that \( \tau \) is strongly abelian.

Recall (Proposition \ref{prop:SA induced algebra}) that each induced algebra
\[ \alg{A}_{|{C_i}} = \inangle{ C_i \; ; \; \left\{ f \in \Pol{}{\alg{A}} \colon f(C_i, \ldots, C_i) \subseteq C_i \right\} } \]is a strongly abelian algebra. For each \(1 \leq i \leq M\), let \( K_i\) be the greatest arity of a decomposition \( \tau \)-boxmap on \( C_i \). Fix operators
\[ d_i(v_1, \ldots, v_{K_i}) = D_i(v_1, \ldots, v_n, \vec{a} ) \]witnessing this; that is, \( d_i \) is a \(K_i\)-ary decomposition operator on \( C_i \) and \(D_i(\vec{x}, \vec{a}) = D_i(\vec{x}, \vec{a}') \) whenever \( \vec{x} \in C_i \) and \( \vec{a} \equiv_\tau \vec{a}' \). This determines a product decomposition
\[ C_i = C_{i,1} \times \cdots \times C_{i,K_i} \]as detailed above.

\begin{definition}
  The multi-sorted first-order language \( \frzflt{L}{\tau} \) will have the following nonlogical symbols:

  For each \( 1 \leq i \leq M \) and each \( 1 \leq j \leq K_i \), the language will have a sort symbol \( \inangle{i,j} \).

  For each \( \tau \)-boxmap 
  \[ f(v_1, \ldots, v_n) = t(v_1, \ldots, v_n, \vec{a}) \colon C_{i_1} \times \cdots C_{i_n} \rightarrow C_{i_0} \](\( \vec{a} \in C_{i_{n+1}} \times \cdots \times C_{i_{n'}} \)) and each \(1 \leq j \leq K_{i_0} \) the language \( \frzflt{L}{\tau} \) will have a basic operation of type declaration
  \[ t^j_{i_1 \cdots i_n i_{n+1} \cdots i_{n'}} \begin{pmatrix} 
    \inangle{i_1,1} & \inangle{i_2,1} & \cdots & \inangle{i_n,1} \\
    \inangle{i_1,2} & \inangle{i_2,2} & \cdots & \inangle{i_n,2} \\
    \vdots & \vdots & \ddots & \vdots \\
    \inangle{i_1,K_{i_1}} & \inangle{i_2,K_{i_2}} & \cdots & \inangle{i_n,K_{i_n}}
  \end{pmatrix} \rightarrow \inangle{i_0,j}. \]
\end{definition}

Note that every term \(t(v_1, \ldots, v_n)\) of \( \alg{A} \) is automatically a \( \tau\)-boxmap when restricted to any product of \(n\) \(\tau\)-classes, so the entire atomic diagram of \( \alg{A} \) is encoded in that of \( \frzflt{\alg{A}}{\tau} \). We will see in a moment that \( \frzflt{\alg{A}}{\tau} \) is strongly abelian; it follows that the language \( \frzflt{L}{\tau} \) may be taken to be finite.

We can characterize terms in this language easily.

\begin{proposition}
  Every term in the language \( \frzflt{L}{\tau} \) is obtained from one of the basic operations \( t^j_{i_1 \cdots i_n}\) by possibly identifying some variables of the same sort.
\end{proposition}

The proof (by induction) is left to the reader.

\begin{construction}
  \begin{enumerate}[ref=\thetheorem(\arabic*)]
  \item We define an \( \frzflt{L}{\tau} \)-structure \( \frzflt{\alg{A}}{\tau} \) analogously to our definition of \( \frz{\alg{A}}{\tau} \) in Construction \ref{conststep:Atau}: each sort
    \[ \inangle{i,j}^{\frzflt{\alg{A}}{\tau}} = C_{i,j} \]Now if \( t^j_{i_1 \cdots i_n i_{n+1} \cdots i_{n'}} \) is a basic operation symbol and \(x_{k,j} \in C_{i_k,j} \) for \( 1 \leq k \leq n \) and \( 1 \leq j \leq K_{i_k} \), set
    \[ x_k = \begin{pmatrix} x_{k,1} \\ x_{k,2} \\ \vdots \\ x_{k,K_{i_k}} \end{pmatrix} \qquad (1 \leq k \leq n) \]and choose any \( \vec{a} \in C_{i_{n+1}} \times \cdots \times C_{i_{n'}} \). Let
    \[ t^{\alg{A}}(x_1, \ldots, x_n, \vec{a}) = x_0 = \begin{pmatrix} x_{0,1} \\ x_{0,2} \\ \vdots \\ x_{0,K_{i_0}} \end{pmatrix} \]It now makes sense to define 
    \[ t^{j}_{i_1 \cdots i_n i_{n+1} \cdots i_{n'}} \begin{pmatrix}
      x_{1,1} & x_{2,1} & \cdots & x_{n,1} \\
      x_{1,2} & x_{2,2} & \cdots & x_{n,2} \\
      \vdots & \vdots & \ddots & \vdots \\
      x_{1,K_{i_1}} & x_{2,K_{i_2}} & \cdots & x_{n, K_{i_n}}
    \end{pmatrix} = x_{0,j}\]

  \item\label{conststep:Btauflt} The foregoing construction generalizes to any \(L\)-structure \( \alg{B} \) having a congruence \( \tau^{\alg{B}} \) such that there exists an isomorphism \( \varphi \colon \alg{A}/\tau \rightarrow \alg{B}/{\tau^{\alg{B}}} \), and such that the same terms \(D_i(v_1, \ldots, v_{K_i}, \ldots, v_{n'}) \) define decomposition \(\tau\)-boxmaps on the classes \( \varphi(C_i) \), with constants taken from the same classes \( \varphi(C_{i_{n+1}}), \ldots, \varphi(C_{i_{n'}})\). (We do not require that no decomposition operator on \( \varphi(C_i) \) have larger arity.)

    Under these hypotheses, each \( \tau^{\alg{B}} \) class \( \varphi(C_i) \) decomposes into a product of \( K_i \) factors as above, and the analogous definition produces a well-defined \( \frzflt{L}{\tau} \)-structure \( \frzflt{\alg{B}}{\tau} \).
  \end{enumerate}
\end{construction}

We state without proof the analogues of the lemmata of Section \ref{sec:frz}, since all the proofs differ only in the bookkeeping:

\begin{proposition}
  Let \( \alg{M} = \frzflt{\alg{B}}{\tau} \) and \( \alg{N} = \frzflt{\alg{C}}{\tau} \). \begin{enumerate}
  \item Let \( \alg{D} \leq \alg{B} \) have nonempty intersection with each \( \tau \)-class; then \( \alg{D} \) satisfies the hypotheses of Construction \ref{conststep:Btauflt}, and \( \frzflt{\alg{D}}{\tau} \) is a substructure of \( \alg{M} \). Moreover, every substructure of \( \alg{M} \) is obtained in this way.

  \item Let \( \theta \leq \tau \) be a congruence on \( \alg{B} \); then \( \alg{B}/\theta \) satisfies the hypotheses of Construction \ref{conststep:Btauflt}, and \( \frzflt{\left( \alg{B}/\theta \right)}{\tau} \) is a homomorphic image of \( \alg{M} \). Moreover, every homomorphic image of \( \alg{M} \) is obtained in this way.

  \item Let \( \alg{D} \leq \alg{B} \times \alg{C} \) be the subalgebra consisting of all pairs \( \vpair{b}{c} \) such that \( \varphi^{-1}(b/\tau) = \varphi^{-1}(c/\tau) \). Then \( \alg{D} \) satisfies the hypotheses of Construction \ref{conststep:Btauflt}, and \( \frzflt{\alg{D}}{\tau} \) is the product of \( \alg{M} \) and \( \alg{N} \) in the sense of \( \frz{L}{\tau} \). (This generalizes to any number of factors.)
  \end{enumerate}
\end{proposition}

\begin{lemma}\label{lemma:frzflt into A}
  \begin{enumerate}
  \item The smallest equationally axiomatizable class containing \( \frzflt{\alg{A}}{\tau} \) is the closure of \( \left\{ \frzflt{\alg{A}}{\tau} \right\} \) under \( \mathrm{HSP} \); this class is axiomatized by the set of all equations which hold in \( \frzflt{\alg{A}}{\tau} \). This variety is finitely axiomatizable.
  \item The sorted structure \( \frzflt{\alg{B}}{\tau} \) is abelian (resp. strongly abelian) if and only if the congruence \( \tau \) was a (strongly) abelian congruence of \( \alg{B} \).
  \item If \( \alg{A} \) belongs to a finitely decidable variety and \( \tau \) is a (strongly) solvable congruence, then \( \HSP{\frzflt{\alg{A}}{\tau}} \) is a (strongly) abelian variety.
  \item If \( \alg{A} \) belongs to any finitely decidable variety and \( \tau \) is either the solvable radical or the strongly solvable radical of \( \alg{A} \), then \( \HSP{\frzflt{\alg{A}}{\tau}} \) semantically interprets into \( \HSP{\alg{A}} \).
  \end{enumerate}
\end{lemma}

\begin{proof}
  The only new statement here is that \( \HSP{\frzflt{\alg{A}}{\tau}} \) is finitely axiomatizable.

  It is well known (e.g. \cite[Theorem 0.17]{MV1989}) that an (ordinary single-sorted) algebra \( \alg{X} \) is strongly abelian if and only if for each term \(t(v_1, \ldots, v_n) \) there exist equivalence relations \(E_1, \ldots, E_n \) on \( X \) such that for all \(x_1, y_1 \ldots, x_n,y_n \in X \),
  \[t(x_1, \ldots, x_n) = t(y_1, \ldots, y_n) \iff \inangle{x_1,y_1} \in E_1, \ldots, \inangle{x_n,y_n} \in E_n. \]Likewise, a congruence \( \tau \) is strongly abelian iff for each term \(t\) and all \(\tau\)-classes
  \[ C_{i_1} \times \cdots \times C_{i_n} \stackrel{t}{\rightarrow} C_0 \]there exist equivalence relations \( E_k \) on \( C_{i_k} \) such that for all \(x_k, y_k \in C_{i_k}\),
  \[ t(x_1, \ldots, x_n) = t(y_1, \ldots, y_n) \iff \inangle{x_1,y_1} \in E_1, \ldots, \inangle{x_n,y_n} \in E_n. \]It follows that such a term action cannot depend on more than \(\log_2(|C_{i_0}|) \) of its variables; in \( \HSP{\frzflt{\alg{A}}{\tau}} \), this means that the basic operation \( t^j_{i_1 \cdots i_n} \) can only depend essentially on at most \(\log_2(|C_{i_0}|) \cdot \max_i K_i\) variables. Since there are only finitely many equations using this many variables, and since \( \HSP{\frzflt{\alg{A}}{\tau}} \) is axiomatized by the subset of these which are true in \( \frzflt{\alg{A}}{\tau} \), we are done.
\end{proof}

\section{Main Semantic Interpretation}\label{sec:main interpretation}

The goal of this section is to prove

\begin{theorem}\label{thm:bipartite into frzflt}
  Let \( \alg{A} \) be a finite algebra in a variety where every strongly solvable congruence is strongly abelian. Let \( \tau \) be the strongly solvable radical of \( \alg{A} \), 
  \begin{equation} C_{i_1} \times \cdots \times C_{i_n} \stackrel{t}{\rightarrow} C_{i_0} \label{eq:boxmap}\end{equation}be any \( \tau \)-boxmap, and let \(K_{i_0}\) be the greatest arity of a decomposition \( \tau \)-boxmap on \(C_{i_0}\). If the map in \eqref{eq:boxmap} depends essentially on more than \(K_{i_0}\) variables, then the class of bipartite graphs interprets semantically into \( \HSP{\frzflt{\alg{A}}{\tau}} \).
\end{theorem}
\newcounter{bipartitethm}
\setcounter{bipartitethm}{\value{theorem}}

The proof of Theorem \ref{thm:main} will be a quick consequence of this.

For the remainder of this section, let \( \alg{A} \) be a fixed finite algebra satisfying the hypotheses of Theorem \ref{thm:bipartite into frzflt}. As before, we choose a fixed enumeration \( C_1, \ldots, C_M \) of the \( \tau \)-classes. Fix decomposition \( \tau \)-boxmaps
\[ d_i(v_1, \ldots, v_{K_i}) = D_i(v_1, \ldots, v_{K_i}, \vec{a}) \colon C_i^{K_i} \rightarrow C_i \]of maximal arity.

\begin{proposition}
  The algebra \( \frzflt{\alg{A}}{\tau} \) is essentially unary if and only if every \( \tau \)-boxmap
  \begin{equation}\label{eq:boxmap2} C_{i_1} \times \cdots \times C_{i_n} \stackrel{t}{\rightarrow} C_{i_0} \end{equation}depends on at most \( K_{i_0} \) variables.
\end{proposition}

\begin{proof}We prove each contrapositive.

  (\(\Rightarrow\)): Let the action of \(t(v_1, \ldots, v_{K_{i_0}+1},\ldots)\) on the box in Equation \eqref{eq:boxmap2} depend essentially on at least the first \(K_{i_0}+1\) variables. Choose a witnessing assignment
  \[ t(a, b_2, \ldots, b_n) \neq t(a', b_2, \ldots, b_n) \]in the first variable: then for some \( 1 \leq j \leq K_{i_0} \),
  \[ t(a, b_2, \ldots, b_n) \not \sim_j t(a', b_2, \ldots, b_n) \]where 
\[ x \sim_j y \iff x = \begin{matrix} x^1 \\ x^2 \\ \vdots \\ x^K \end{matrix}, y = \begin{matrix} y^1 \\ y^2 \\ \vdots \\ y^K \end{matrix} \text{ and } x^j = y^j \]
For this \(j\), the term \( t^j_{i_1 \cdots i_n}\) depends on one of the variables in its first column. Similarly, for each of the variables \(v_2, \ldots, v_{K_{i_0}+1} \) one of the terms \( t^j_{i_1 \cdots i_n}\) depends on a variable in the corresponding column. Now use the pigeonhole principle to get one of the \( t^j_{i_1 \cdots i_n}\) depending on at least two variables.

  (\( \Leftarrow \)): We first claim that if \( t^j_{i_1 \cdots i_n} \) depends in \( \frzflt{\alg{A}}{\tau} \) on the variable in column \(c \) and row \(r\), then in \(\alg{A}\) the operation
  \[ d_{i_0} \left( y_1, \ldots, y_{j-1}, t \left(d_{i_1} \left(x_1^1, \ldots, x_1^{K_{i_1}}\right), \ldots, d_{i_n}\left(x_n^1, \ldots, x_n^{K_{i_n}}\right)\right), \ldots, y_{K_{i_0}} \right) \]depends on \(x_c^r\) (as well as on each of the \(y\)s).

  To see this, pick a witnessing package
  \[ t^j_{i_1 \cdots i_n} \begin{pmatrix}
    b_1^1 & b_2^1 & \cdots & b_n^1 \\
    b_1^2 & b_2^2 & \cdots & b_n^2 \\
    \vdots & \vdots & a & \vdots \\
    b_1^{K_{i_1}} & b_2^{K_{i_2}} & \cdots & b_n^{K_{i_n}}
  \end{pmatrix} \neq t^j_{i_1 \cdots i_n} \begin{pmatrix}
    b_1^1 & b_2^1 & \cdots & b_n^1 \\
    b_1^2 & b_2^2 & \cdots & b_n^2 \\
    \vdots & \vdots & a' & \vdots \\
    b_1^{K_{i_1}} & b_2^{K_{i_2}} & \cdots & b_n^{K_{i_n}}
  \end{pmatrix} \]Upstairs in \( \alg{A} \) this becomes
  \[ t(b_1, \ldots, \hat{a}, \ldots, b_n) \not \sim_j t(b_1, \ldots, \hat{a}', \ldots, b_n) \]which is what we need.

  Now, let \(s\) is any term of \( \frzflt{L}{\tau} \) which depends in \( \frzflt{\alg{A}}{\tau} \) on two of its variables. Without loss of generality, we may take s to be equal to \(t^j_{i_1 \cdots i_n}\), since identification of variables can never increase essential arity. Let \(s\) depend on \(v_c^r, v_{c'}^{r'} \); then the term
  \[ d_{i_0} \left( y_1, \ldots, y_{j-1}, t \left(d_{i_1} \left(x_1^1, \ldots, x_1^{K_{i_1}}\right), \ldots, d_{i_n}\left(x_n^1, \ldots, x_n^{K_{i_n}}\right)\right), \ldots, y_{K_{i_0}} \right) \] depends on all the \(y\) variables and \(x_c^r, x_{c'}^{r'} \).
\end{proof}

\begin{lemma}\label{lemma:leftinversesunary}
  Let \(t(v_1, v_2, \ldots, v_n) \) be an \( \frzflt{L}{\tau} \)-term.
  \begin{enumerate}
  \item If the formula \(  t(v_1, v_1, v_3, \ldots, v_n) = v_1 \) is well-formed and holds universally in \( \frzflt{\alg{A}}{\tau} \), then \(t\) is essentially unary in \( \frzflt{\alg{A}}{\tau} \).
  \item If for some terms \(s_k(v_1, v_2, \ldots)\), the formula 
    \[  t(s_1(\vec{v}), s_2(\vec{v}), \ldots, s_n(\vec{v})) = v_1 \]is well-formed and holds universally in \( \frzflt{\alg{A}}{\tau} \) (in which case we call \(t\) right-invertible) then \(t\) is essentially unary in \( \frzflt{\alg{A}}{\tau} \).
  \end{enumerate}
\end{lemma}

\begin{proof}
  \begin{enumerate}[ref=(\arabic*)]
  \item\label{lemitem:xxyx} For any \( y_3, y_3', \ldots, y_n, y_n' \) in the appropriate sorts, the ranges of the polynomials
    \[ t(v_1, v_2, \vec{y}), \qquad t(v_1, v_2, \vec{y'}) \]are not disjoint. Since \( \frzflt{\alg{A}}{\tau} \) is strongly abelian, all such polynomials must in fact be equal.

    Let \(t\) be a specialization of \(s^j_{i_1 i_2 \cdots} \) for some term \(s(x_1, x_2, \ldots) \) in \(L\). Since \( v_1, v_2\) have the same sort as \(t\), we may as well assume that \(v_1\) represents the \(j\) coordinate of \(x_1\), and similarly for \(v_2\). The operation
    \begin{align*} &g(y_1,y_2, \ldots, j_{j-1},x_1,x_2,\ldots,y_{j+1},\ldots,y_{K_{i_0}}) \\ &\quad = \\& d_{i_0}(y_1,y_2, \ldots, y_{j-1}, s(x_1,x_2, \ldots), y_{j+1}, \ldots, y_{K_{i_0}}) \end{align*}then depends only on the variables shown (i.e.~not on \(x_3,\ldots\)) as a function on
    \[ \underbrace{C_{i_0} \times \cdots \times C_{i_0}}_{j-1} \times C_{i_0} \times C_{i_0} \times C_{i_3} \times \cdots \times C_{i_\ell} \times \underbrace{C_{i_0} \times \cdots \times C_{i_0}}_{n-j} \rightarrow C_{i_0}\]and is idempotent on the variables in sort \( \inangle{i_0,j} \). Hence \( \alg{A} \) has a \(\tau\)-boxmap
    \[ g(y_1, \ldots, y_{j-1},x_1,x_2,y_{j+1},\ldots,y_{K_{i_0}}) \]which is an idempotent operation on \( C_{i_0} \) and depends on all the \(y_k\). By maximality this operation cannot depend on both \(x_1\) and \(x_2\), implying that \(t\) did not depend on both \(v_1\) and \(v_2 \) to begin with.

  \item Let \(v_1^1, \ldots, v_1^n \) be variables of the first input sort of \(s\). By part \ref{lemitem:xxyx}, the term
    \[ t(s_1(v^1_1,v_2,\ldots), s_2(v^1_1,v_2, \ldots), \ldots, s_{n-1}(v^1_1,v_2, \ldots), s_n(v^n_1,v_2,\ldots)) \]depends on none of \(v_2, \ldots, v_n\) and on only one of \(v^1_1, v^n_1\). Proceeding inductively, we see that
    \[\hat{t}(v_1^1,v_1^2,\ldots,v_1^n,v_2,\ldots) = t(s_1(v_1^1,v_2, \ldots),s_2(v_1^2,v_2,\ldots),\ldots,s_n(v_1^n,v_2,\ldots)) \]depends on just one variable, say \( v_1^1 \), and in fact \[\hat{t}(v_1^1,v_1^2,\ldots,v_1^n,v_2,\ldots) = v_1^1.\]

    We claim that \(t\) depends only on its first variable. To see this, let \( a_1, a_2, a_2', \ldots, a_n, a_n', \vec{b} \) be any elements of the appropriate sorts. Define \begin{align*}
      u &= t(a_1, a_2, \ldots, a_n) \\
      u'&= t(a_1, a_2', \ldots, a_n') \\
      q_2  &= s_2(a_2, \vec{b}) \\
      q_2' &= s_2(a_2',\vec{b}) \\
      &\vdots \\
      q_n' &= s_n(a_n', \vec{b})
    \end{align*}Then since the ranges of \(t(v_1, a_2, \ldots, a_n) \) and \(t(v_1, q_2, \ldots, q_n) \) both contain \(u\), these two polynomials must be equal; likewise the polynomials \(t(v_1, a_1', \ldots, a_n') \) and \( t(v_1, q_2', \ldots, q_n') \). But \begin{align*}
      u = t(s_1(u,\vec{b}),q_2, \ldots, q_n) &= t(s_1(u,\vec{b}),q_2', \ldots, q_n') \\ &\Downarrow \\
      t(v_1,q_2, \ldots, q_n) &= t(v_1, q_2', \ldots, q_n')
    \end{align*}which shows that \begin{align*}
      t(v_1, a_2, \ldots, a_n) &= t(v_1, q_2, \ldots, q_n) \\
      &= t(v_1, q_2', \ldots, q_n') \\
      &= t(v_1, a_2', \ldots, a_n') \end{align*}Since \(a_k, a_k'\) were arbitrary, we are done.
  \end{enumerate}
\end{proof}

\begin{lemma}\label{lemma:essentially binary and not left invertible}
  If \( \frzflt{\alg{A}}{\tau} \) is not essentially unary, then there is an \(\frzflt{L}{\tau}\)-term depending essentially in \( \frzflt{\alg{A}}{\tau} \) on at least two variables and not left-invertible at any.
\end{lemma}

\begin{proof}
  We show how to take a term depending essentially on \(v_1, v_2\) and invertible at \(v_1\), and produce a new term depending essentially on \(v_2\) and at another variable \(v_0\) (possibly of a different sort than \(v_1\)) and not invertible at \(v_0\). We will then show that if we started with a term which was not left-invertible at \(v_2\), then the new term we construct still has this property.

  Assume that \(t(v_1, v_2, \ldots, v_n) \) depends essentially on \(v_1\) and \(v_2\), and that
  \begin{equation}\label{eq:v1 in ran(s)} s(t(v_1, \ldots, v_n), v_{n+1}, \ldots) = v_1 \end{equation}The same logic used in part \ref{lemitem:xxyx} of Lemma \ref{lemma:leftinversesunary} guarantees that \(s\) cannot depend on any variable except the first, so we will write \(s(x)\) as if it were a unary term.

  Let
  \[ \hat{t}(v_0, v_2, \ldots, v_n) = t(s(v_0),v_2,\ldots, v_n) \]Since \(s\) maps the output sort of \(t\) onto the sort of \(v_1\) in \(\frzflt{\alg{A}}{\tau} \), this new term \(\hat{t}\) must depend essentially on \(v_0\) and \(v_2\).

  \begin{claim}\label{claim:t hat}
    \(\hat{t}\) is not left-invertible at \(v_0\).
  \end{claim}

  Suppose elsewise: let
  \[ r(\hat{t}(v_0,v_2,\ldots,v_n)) = v_0 \]Define another term
  \[ q(v_0,v_2,v_2',\vec{w}) = \hat{t}(\hat{t}(v_0,v_2,\vec{w}),v_2',\vec{w}) \](where \( \vec{w} = v_3, \ldots, v_n \)). Then on the one hand \begin{align*}
    \hat{t}(v_0,v_2,\vec{w}) &= r(\hat{t}(\hat{t}(v_0,v_2,\vec{w}),v_2',\vec{w})) \\
    &= r(q(v_0,v_2,v_2',\vec{w}))
  \end{align*}so \(q\) must depend essentially on \(v_2\). But on the other hand \begin{align*}
    q(v_0,v_2,v_2',\vec{w}) &= \hat{t}(\hat{t}(v_0,v_2,\vec{w}),v_2',\vec{w}) \\
    &= t(s(\hat{t}(v_0,v_2,\vec{w})),v_2',\vec{w}) \\
    &= t(s(t(s(v_0),v_2,\vec{w})),v_2',\vec{w}) \\
    &= t(s(v_0),v_2',\vec{w})
  \end{align*}which does \emph{not} depend on \(v_2\).\pfend{\ref{claim:t hat}}

  Lastly, we must show that if \(\hat{t}\) were left-invertible at \(v_2\) then \(t\) would already have been. This is not hard: suppose
  \[ v_2 = r_2(\hat{t}(v_0,v_2,\ldots, v_n)) = r(t(s(v_0),v_2,\ldots, v_n)) \]Again using the logic of part \ref{lemitem:xxyx} of lemma \ref{lemma:leftinversesunary}, the term \[r(t(s(v_0), v_2, \ldots, v_n)) \]can only depend on \(v_2\); since by Equation \eqref{eq:v1 in ran(s)}, \(v_1 \in \mathrm{ran}(s) \) (considered as elements of the free algebra \( \alg{F}_{\var{V}(\frzflt{\alg{A}}{\tau})}(v_0,v_1,v_2,\ldots) \)), we must have that \(r_2\) inverts \(t\) as well.
\end{proof}

\begin{construction}\label{const:F(X)/theta}
  Let \(X\) be any sorted family of generators for a free algebra \( \alg{F} = \alg{F}(X) \) in \(\var{V}(\frzflt{\alg{A}}{\tau}) \). Let \(f_0\) be an arbitrary fixed element of \(F\), and let \(\alg{F}' = \alg{F}(X \cup \{z\})\), where \(z\) is a new free generator of the same sort as \(f_0\).

  Generate a congruence \( \theta \in \Con{\alg{F}'}\) from all pairs
  \[ \inangle{t(f_0, \vec{u}), t(z,\vec{u})} \]such that \( \vec{u} \in F \) and \(t(v_0, \vec{v}) \) is not left-invertible at \(v_0 \). (Observe that if a term \(g \in F\) occurs as the second member \(t(z, \vec{u})\) of such a pair, by freeness we get that \(t\) does not depend on its first variable, so that the pair is in fact trivial.)
\end{construction}

\begin{lemma}
  Let \( \alg{F}, \alg{F}',\) and \(\theta \) be as in Construction \ref{const:F(X)/theta}. If \(a \in F\) and \( a \equiv_\theta b\), then either \(a = b \) or \( \langle a,b \rangle \) is a generating pair.
\end{lemma}

\begin{proof}
  Suppose we have basic nontrivial \(\theta\)-links \( a \text{---} c \text{---} b \), where
  \[ \langle a, c \rangle = \langle t_1(f_0, \vec{u}_1), t_1(z, \vec{u}_1) \rangle \]

  Case 1:
  \[ \langle c, b \rangle = \langle p_2(t_2(f_0, \vec{u}_2)), p_2(t_2(z,\vec{u}_2)) \rangle \]where \(p_2(v_0) = g_2(v_0,z,\vec{w}_2) \in \Pol{1}{\frzflt{\alg{A}}{\tau}} \) for some terms \(g, \vec{w} \in F\).

  We have 
  \[ c = t_1(z,\vec{u}_1) = g_2(t_2(f_0,\vec{u}_2),z, \vec{w}_2) \]and since \(z\) is a free generator, we may substitute any term for \(z\) in the above equation. In particular, \begin{align}
    a = t_1(f_0, \vec{u}_1) &= g_2(t_2(f_0, \vec{u}_2),f_0, \vec{w}_2) \notag\\
    b &= g_2(t_2(z, \vec{u}_2),z, \vec{w}_2) \label{eq:acb->ab}
  \end{align}We will be done with Case 1 if we can establish
  
  \begin{claim}\label{claim:g2 not invertible case1}
    \( g_2(t_2(v_0, \vec{u}_2),v_0,\vec{w}_2) \) is not left-invertible at \(v_0\).
  \end{claim}

  Suppose the contrary, say
  \begin{equation}\label{eq:rg2} r(g_2(t_2(v_0, \vec{u}_2),v_0,\vec{w}_2)) = v_0 \end{equation}By Lemma \ref{lemma:leftinversesunary}, the term
  \[ r(g_2(t_2(v_0, \vec{u}_2),v_1,\vec{w}_2)) \]must depend only on \(v_0\) or \(v_1\), and because of Equation \eqref{eq:rg2} must project to the active variable. Moreover, it cannot be \(v_0\), since then this would be a left-inversion of \(t_2(v_0, \vec{u}_2) \). But if \(v_1\) were the active variable, we would have
  \[ v_1 = r(g_2(t_2(v_0,\vec{u}_2),v_1,\vec{w}_2)) = r(g_2(t_2(f_0,\vec{u}_2),v_1,\vec{w}_2)) = r(t_1(v_1,\vec{u}_1)) \]contradicting our assumption that \(t_1(v_1,\vec{u}_1)\) was not invertible.\pfend{\ref{claim:g2 not invertible case1}}

  Now Equation \eqref{eq:acb->ab} shows that \( \langle a,b \rangle \) is a generating pair.

  \vspace{10pt}
  Case 2: As before, 
  \[ \langle a,c \rangle = \langle t_1(f_0,\vec{u}_1), t_1(z, \vec{u}_1) \rangle \]but now
  \[ \langle c,b \rangle = \langle p_2(t_2(z, \vec{u})), p_2(t_2(f_0,\vec{u}_2)) \rangle \]with \(p_2\) a polynomial as before. Since
  \[ c = t_1(z, \vec{u}_1) = g_2(t_2(z, \vec{u}_2), z, \vec{w}_2) \]and \(z\) is a free generator, the same equation holds under any substitution for \(z\): \begin{align*}
    a = t_1(f_0, \vec{u}_1) &= g_2(t_2(f_0,\vec{u}_2), f_0, \vec{w}_2) \\
    b &= g_2(t_2(f_0,\vec{u}_2), z, \vec{w}_2) 
  \end{align*}As before, the following claim suffices:

  \begin{claim}\label{claim:g2 not invertible case2}
    \(g_2(t_2(f_0, \vec{u}_2), v_0, \vec{w}_2) \) is not left-invertible at \(v_0\).
  \end{claim}

  If it were, so
  \[ r(g_2(t_2(f_0,\vec{u}_2),v_0, \vec{w}_2)) = v_0 \]then the range of this polynomial contains the whole sort of \(f_0\). In particular, 
  \[ r(c) \in \mathrm{ran}\left( r(g_2(t_2(z, \vec{u}_2), \bullet, \vec{w}_2)) \right) \cap \mathrm{ran}\left( r(g_2(t_2(f_0, \vec{u}_2), \bullet, \vec{w}_2)) \right) \]By strong abelianness, the two polynomials in the above equation should be equal, contradicting our original assumptions.\pfend{\ref{claim:g2 not invertible case2}}

\end{proof}

\begin{proposition}\label{prop:zisolated}
  Let \( \alg{F}, \alg{F}',\) and \(\theta \) be as in Construction \ref{const:F(X)/theta}. Then \(z\) is isolated\( \pmod\theta\).
\end{proposition}

\begin{proof}
  Let \( \{ z, x \} = \{ p(t(f_0, \vec{u})), p(t(z,\vec{u})) \} \) be a basic \(\theta\)-pair, where \( p(v_0) = g(v_0, z, \vec{w}) \) as in the previous lemma.

  First suppose \[ z = p(t(z, \vec{u}))  = g(t(z,\vec{u}),z,\vec{w})\]Then by Lemma \ref{lemma:leftinversesunary}, \(( g(t(v_0, \vec{u}),v_1, \vec{w}) \) depends only on one variable, either \(v_0\) or \(v_1\). Moreover, \(v_0\) is not a possibility, since then \(t\) would be left-invertible. We conclude that \(g(t(v_0, \vec{u}),v_1, \vec{w}) = v_1 \) throughout \( \HSP{\frzflt{\alg{A}}{\tau}} \).

  Next suppose \[ z = p(t(f_0, \vec{u})) = g(t(f_0, \vec{u}),z, \vec{w}) \]Then \(g(t(v_0, \vec{u}),v_1,\vec{w}) \) is right-invertible; invoking Lemma \ref{lemma:leftinversesunary} again, this term is essentially unary, and since \( f_0 \in F\) and \(z\) is not, the dependency must be on \(v_1\); hence 
  \[ g(t(v_0,\vec{u}),v_1,\vec{w}) = v_1 \]is valid in \( \HSP{\frzflt{\alg{A}}{\tau}} \).

  In either case, we conclude that \(z = x\).
\end{proof}

The content of the previous two lemmas is that, for \( \alg{F}, \alg{F}', f_0\), and \(\theta\) defined in this way, and for \( \alg{C} = \alg{F}'/\theta\), we have that \(\alg{F} \) is an isomorphic substructure of \( \alg{C} \), and \( f_0 \) and \( z \) are indistinguishable by the action of non-left-invertible terms \(t( \bullet, \vec{u}) \) taken from \(F\).

Recall that since \( \frzflt{\alg{A}}{\tau} \) is strongly abelian, there is an upper bound on the essential arity of terms over this algebra. (For example, \( |A| \cdot \max_i K_i \) would work.) Let \(T\) be a finite set of \( \frzflt{L}{\tau} \) terms such that every term operation of \( \frzflt{\alg{A}}{\tau} \) is given (up to renaming of variables) by one of the terms in \(T\).

For each sort \( \langle i, j \rangle \), let \( N_{\langle i,j \rangle} \subset T \) be the set of all terms \(t(v_0, v_1, \ldots) \) such that \(v_0\) has sort \( \langle i, j \rangle \) and \(t\) is not left-invertible at \(v_0\). Then the relations
\[ a \propto_{\langle i,j \rangle} b \iff \bigwedge_{t \in N_{\langle i,j \rangle}} \forall \vec{u} \; t(a,\vec{u}) = t(b,\vec{u}) \]together comprise a definable equivalence relation on any \( \alg{M} \in \HSP{\frzflt{\alg{A}}{\tau}} \). We will usually write \( a \propto b \) instead of \( a \propto_{\langle i, j \rangle} b\).

It is clear from the definition that \( \mathbf{a} \propto \mathbf{b} \) in a product \( \prod_{x \in X} \alg{B}_x \) if and only if \(a^x \propto b^x\) in each stalk.

\begin{proposition}\label{prop:RinvPreservesPropto}
  If \(s(v_0, v_1, \ldots,v_n) \) is a right-invertible term depending only on \(v_0\), then for any \( \alg{M} \in \HSP{\frzflt{\alg{A}}{\tau}} \), any \( a \propto b \in M \), and any \( x_1,  \ldots, x_n \in M\) of the appropriate sorts, \( s(a,x_1, \ldots x_n) \propto s(b, x_1, \ldots x_n) \).
\end{proposition}

\begin{proof}
  Say the sort of \(v_0\) is \(\langle i,j \rangle \). Let
  \[s(t_0(y, \vec{z}),v_1, \ldots, v_n) = y \]and let \(t(v_0, \ldots, v_\ell) \in N_{\langle i,j \rangle}\). It suffices to show that
  \[t(s(v_0,\ldots, v_n),v_{n+1},\ldots, v_{n+\ell}) \in N_{\langle i,j \rangle}\]too.

  Suppose otherwise: then for some essentially unary term \(r(v_0, \ldots) \) we have
  \[ r(t(s(v_0,\ldots, v_n),v_{n+1},\ldots, v_{n+\ell})) = v_0 \]Then
  \begin{align*} t_0(y, \vec{z}) &= r(t(s(t_0(y,\vec{z}),\ldots, v_n),v_{n+1},\ldots, v_{n+\ell})) \\ 
    &= r(t(y,v_{n+1},\ldots, v_{n+\ell})) \\
    y &= s(t_0(y, \vec{z}),v_1, \ldots, v_n) \\
    &= s(r(t(y,v_{n+1},\ldots, v_{n+\ell})),v_1, \ldots, v_n)
  \end{align*}contradicting our assumption that \(t\) was not left-invertible at its first variable.
\end{proof}

We are ready for:
  \newcounter{realthm}
  \setcounter{realthm}{\value{theorem}}
  \newcounter{realclaim}

\begin{proof}[Proof of Theorem \ref{thm:bipartite into frzflt}]
  \setcounter{theorem}{\value{bipartitethm}}
  Let \( \alg{A} \) be a finite algebra with strongly solvable radical \( \tau \) such that every strongly solvable congruence in \( \HSP{\alg{A}} \) is strongly abelian, and suppose that \( \frzflt{\alg{A}}{\tau} \) is not essentially unary. By Lemma \ref{lemma:essentially binary and not left invertible}, we may fix a term \(q(v_1,\ldots,v_\ell)\) depending essentially on \(v_1,v_2\) but not left-invertible at either. Let \(X\) be a sorted collection of free generators: one \(x_{\inangle{i,j}} \) for each sort \( \inangle{i,j} \), as well as two generators \( a_0, a_1\) of the sort of \(v_1\) and two more \(b_0,b_1\) of the sort of \(v_2\). Let
  \[ v_1 * v_2 = q(v_1,v_2,x_{\inangle{i_3,j_3}},\ldots,x_{\inangle{i_\ell,j_\ell}}) \in \Pol{2}{\alg{F}(X)} \]and define elements
  \begin{align*}
  0 &= a_0 * b_0 \\
  1 &= a_0 * b_1 \\
  2 &= a_1 * b_0 \\
  3 &= a_1 * b_1
\end{align*}(These elements are all distinct since \(*\) depends on both variables.) Let \( \langle i_0, j_0 \rangle \) be the type of these four elements, and let \( \alg{C} = \alg{F}'/\theta \), where \( \alg{F}' \) and \(\theta \) are built according to Construction \ref{const:F(X)/theta}, with \(0\) playing the role of \(f_0\). As we remarked before, \( \alg{F} \leq \alg{C} \).

  We first observe that, by construction, for any \( t(v_0, \ldots, v_n) \in N_{\inangle{i_0,j_0}} \) and any \( \vec{u} \in F\),
  \[ \alg{C} \models t(0,\vec{u}) = t(z,\vec{u}) \]Since \( \alg{C} \) is strongly abelian, it follows that the polynomials \(t(0, \bullet) \) and \(t(z,\bullet)\) are equal: that is,
  \[ \alg{C} \models 0 \propto z. \]

  \begin{claim}\label{claim:0123 not propto}
    \( \{0,1,2,3\} \) are pairwise \(\propto\)-inequivalent.
  \end{claim}
  
  We will show that \( 0 \not \propto 1 \); the remaining cases are similar.

  Suppose for the sake of contradiction that \( 0 \propto 1\). Observe that \(0 \propto 1\) in \( \alg{F} \) also.
  
  \begin{subclaim}\label{subclaim:3 isolated}Under the hypothesis that \( 0 \propto 1 \), \(3\) is isolated modulo \( \beta = \Cg{\alg{F}}{0}{1} \).
  \end{subclaim}
  
  To see this, let \( 3 \in \{g(0, \vec{u}),g(1,\vec{u})\} \) for some term \(g\). Then we have
  \[ 3 = a_1 * b_1 = g(a_0*b, \vec{u}) \]for \(b \) either \(b_0\) or \(b_1\); since \(a_0\) appears on the right but not the right and \( \alg{F} \) is free,
  \[ g(a_0*b,\vec{u}) = g(a_1*b,\vec{u}) \]Thus the polynomial is not injective on \(\inangle{i_0,j_0}\), so \(g(v_0,\vec{u})\) cannot be left-invertible, and hence belongs to \(N_{\inangle{i_0,j_0}}\).

  Our assumption that \(0 \propto 1\) now forces \(g(0,\vec{u})\) to be equal to \(g(1,\vec{u})\).\hspace{\stretch{1}}\(\dashv_{\ref{subclaim:3 isolated}}\)

  In particular, \( 2 \not \equiv_\beta 3 \). But then
  \begin{align*}
    a_0 * b_0 &\equiv_\beta a_0 * b_1 \\ &\text{but} \\
    a_1 * b_0 &\not\equiv_\beta a_1 * b_1
  \end{align*}so \( \beta \) is not abelian. This is a contradiction; the remaining five cases are proved analogously.\hspace{\stretch{1}}\(\dashv_{\ref{claim:0123 not propto}}\)

  Our plan is to semantically interpret the class of bipartite graphs without isolated vertices into \( \HSP{\frzflt{\alg{A}}{\tau}} \). (It is well known that the theory of bipartite graphs is computably inseparable from the set of sentences false in some finite bipartite graph.) Our strategy will be to define an algebra \( \alg{D}(\mathbb{G}) \) for each graph \( \mathbb{G} \), and then to show that certain relations are uniformly first-order definable in these algebras. (Here ``uniformly'' means that the respective relations are defined via the same first-order formulas for all \( \alg{D}(\mathbb{G}) \); the subsets defined by these formulas in algebras in \( \HSP{\frzflt{\alg{A}}{\tau}} \) but not of the form \( \alg{D}(\mathbb{G}) \) may be quite strange and bear no resemblance to the relations we intend.)

  For us, a bipartite graph will be a two-sorted structure \( \mathbb{G} = \inangle{R^{\mathbb{G}},B^{\mathbb{G}} \; ; \; E^{\mathbb{G}}} \), where \(E\) has type signature \(\inangle{R,B}\).

  \setcounter{theorem}{\value{realthm}}
  \setcounter{realclaim}{\value{claim}}
  \begin{construction}
    Let \( \mathbb{G} \) be a bipartite graph. We define a subpower \( \alg{D} = \alg{D}(\mathbb{G}) \leq \alg{C}^\Gamma \) as follows: the index set \( \Gamma = R^{\mathbb{G}} \sqcup B^{\mathbb{G}} \sqcup \{ \clubsuit, \spadesuit\} \), and \(\alg{D} \) is generated by all points
    \begin{align*} 
      \iota_x = x_{|\Gamma} & \qquad (x \in X) \\
      \chi_v = {a_1}_{|v} \oplus {a_0}_{|\mathrm{else}} & \qquad (v \in R^{\mathbb{G}}) \\
      \chi_v = {b_1}_{|v} \oplus {b_0}_{|\mathrm{else}} & \qquad (v \in B^{\mathbb{G}}) \\
      \chi_{e,\clubsuit} = 2_{|v} \oplus 1_{|w} \oplus z_{|\clubsuit} \oplus 0_{|\mathrm{else}} & \qquad (e = \langle v, w \rangle \in E^{\mathbb{G}})\\
      \chi_{e,\spadesuit} = 2_{|v} \oplus 1_{|w} \oplus z_{|\spadesuit} \oplus 0_{|\mathrm{else}} & \qquad (e = \langle v, w \rangle \in E^{\mathbb{G}})
    \end{align*}We let
\begin{align*} 
  \chi_R &= \{ \chi_v : v \in R^{\mathbb{G}} \} \\
  \chi_B &= \{ \chi_v : v \in G^{\mathbb{G}} \} \\
  \chi_E &= \{ \chi_{e,\clubsuit}, \chi_{e,\spadesuit} : e \in E^{\mathbb{G}} \}
\end{align*}
  \end{construction}
  \setcounter{realthm}{\value{theorem}}
  \setcounter{theorem}{\value{bipartitethm}}
  \setcounter{claim}{\value{realclaim}}

  By abuse of notation, \(X\) will still denote the set of diagonal generators \( \iota_x \). We will suppose that we have constant symbols for all the \( \iota_x \), so that \(X\) (and hence \(F\), the subalgebra generated by \(X\)) is a uniformly definable subset of \( D \).

  Note that \(\alg{D} \) is not quite a diagonal subpower; it contains all diagonal elements from \( \alg{F} \), but none of those from \( \alg{C}\setminus \alg{F} \).

  \begin{claim}\label{claim:generators initial}
    If for some term \(t\) and elements \(\vec{\point{x}}\) of \( \alg{D} \), \(t(\point{x}_1, \ldots, \point{x}_n)\) is equal to one of the non-diagonal generators, then \(t\) is right-invertible (and hence essentially unary).
  \end{claim}

  Suppose first that \(t(\point{x}_1, \ldots, \point{x}_n) = \chi_v  \in \chi_R\). Then
  \[ a_1 = \chi_v^v = t(x_1^v, \ldots, x_n^v) \]and all the elements in this equality belong to \(F\). Since \(\alg{F}\) is free, this is precisely the statement that \(t\) is right-invertible.

  The case where \(v\) is a blue vertex is the same.

  Next let \( t(\point{x}_1, \ldots, \point{x}_n) = \chi_{e,\clubsuit}\). Then
  \[ z = (\chi_{e,\clubsuit})^\clubsuit = t(x_1^\clubsuit, \ldots, x_n^\clubsuit) \]so that in \( \alg{F}'\), \(z \equiv_\theta t(y_1, \ldots, y_n) \) for \(y_k / \theta = x_k^\clubsuit \). By Proposition \ref{prop:zisolated}, \[t(y_1, \ldots, y_k) = z\]once again showing that \(t\) is right-invertible.\pfend{\ref{claim:generators initial}}

  The set \( \NRINV \subset D \) of all \(x \) such that
  \begin{quotation}
    \( x \) is neither diagonal nor in the image of any term which is not right-invertible.
  \end{quotation}is uniformly first-order, and we have just shown that every off-diagonal generator lies in this set. While it would be nice if this were actually the set of off-diagonal generators, this might be too much to ask.

  To get around this, define \( x \leq y \) in \( \alg{D} \) if for some essentially unary term \(t(v_0, \ldots) \) we have \(x = t^{\alg{D}}(y, \ldots)\). Then \( \leq \) is a definable preorder, and its associated partial order \( \sim \) is of course definable too, as is the property of being in a maximal \( \sim \)-equivalence class.

  \begin{claim}\label{claim:sim separates NRINV}
    The map \( \chi \mapsto \chi/\sim \) is a bijection of off-diagonal generators to \( \leq \)-maximal \( \sim \)-classes containing a member of \( \NRINV \).
  \end{claim}

  To prove this, we must first show that no two distinct off-diagonal generators are \( \leq \)-related. This is done by exhaustive case analysis; none of the cases are hard, but there are a lot of them. We show two, and leave the rest to the skeptic.

  For our first model case, suppose \(v \) is a red vertex and \( \chi_v \leq \chi_{e,\clubsuit} \) for some \(e\). Then for some essentially unary term \(t(v_0, \ldots) \),
  \begin{align*}
    \chi_v &= t(\chi_{e,\clubsuit}) \\
    a_0 = \left(\chi_v\right)^\spadesuit &= t\left(\left(\chi_{e,\clubsuit}\right)^\spadesuit,\ldots \right) = t(0,\ldots) = t(a_0 * b_0, \ldots)
  \end{align*}Since \(a_0,b_0\) were free generators, this would imply that the operation \(v_0 * v_1\) is left-invertible at \(v_0\), a contradiction.

  Next suppose \( \chi_{e,\clubsuit} \leq \chi_v \). Then
  \begin{align*}
    \chi_{e,\clubsuit} &= t( \chi_v, \ldots) \\
    z = \left( \chi_{e,\clubsuit} \right)^\clubsuit &= t\left( \left( \chi_v\right)^\clubsuit, \ldots \right) \in F
  \end{align*}a contradiction. The rest of the cases are handled similarly.

  So we have that if we have generators \( \point{x}_1 \leq \point{x}_2 \) then \( \point{x}_1 = \point{x}_2 \). Now: suppose that \( \point{y} \in \NRINV \). We have \(\point{y} = t(\point{x}_1, \ldots, \point{x}_n) \) for some term \(t\) and some generators \( \point{x}_k \). But by assumption, \(t\) is right-invertible, hence depends only on one variable (say the first). In other words \( \point{y} \leq \point{x}_1 \). Hence every maximal \(\sim\)-class containing a member of \( \NRINV \) contains a generator.

  Lastly, if \( \point{x}_0 \) is an off-diagonal generator and \( \point{x}_0 \point{y} \in \NRINV \), then \( \point{x}_0 \leq \point{y} \leq \point{x}_1 \) for some generator \( \point{x}_1 \). By the previous part, \( \point{x}_0 = \point{x}_1 \). This shows that the \( \sim \)-class of every off-diagonal generator is maximal.
  \pfend{\ref{claim:sim separates NRINV}}

  Let \( \GEN \) be the set of all elements of \( D \) \(\sim\)-equivalent to an off-diagonal generator. As we have just seen, this set is uniformly definable: \( \point{y} \in \GEN \) if and only if
  \begin{quotation}
    \(\point{y} \in \NRINV\) and for all \( \point{y}' \in \NRINV \), \( \point{y} \leq \point{y}' \rightarrow \point{y}' \leq \point{y} \).
  \end{quotation}

  We want to be able to distinguish between edge-type and vertex-type generators. To do this, first observe that for any edge \(e\), \( \chi_{e, \clubsuit} \propto \chi_{e, \spadesuit} \) since the relation holds in every factor. This prompts us to set \( \EDGEGEN \) to be the subset of \( \GEN \) consisting of all \( \point{x} \) such that
  \begin{quotation}
    There exist \( \point{x}', \point{y} \in \GEN \) with \( \point{x} \sim \point{x}'\), \( \point{x} \not \sim \point{y} \), and \( \point{x}' \propto \point{y} \). 
  \end{quotation}This set is clearly definable.

  \begin{claim}\label{claim:EDGEGEN defines edge gens}
    For \( \point{y} \in \GEN \), \( \point{y} \in \EDGEGEN \) if and only if the (unique) generator in \( \point{y} / \sim \) has edge type.
  \end{claim}
  
  Proof: By construction, each \( \chi_{e, \clubsuit} \) and each \( \chi_{e, \spadesuit} \) belong to \( \EDGEGEN \). Also, \( \EDGEGEN \) is clearly a union of \( \sim \)-classes.

  Hence it suffices to show that \( \chi_v \notin \EDGEGEN \) for any vertex \(v\). Suppose this were false: then we would have elements \( \point{x} \sim \chi_v \) and \( \point{y} \not \sim \chi_v\) with \( \point{x} \propto \point{y} \). Let \( \gamma \) be the generator \( \sim \)-equivalent to \( \point{y} \).
  
  Since \( \point{x} \sim \chi_v\), they are connected by essentially unary terms
  \[ \point{x} = f_1(\chi_v) \qquad \chi_v = f_2(\point{x}) \]and likewise
  \[ \point{y} = g_1(\gamma) \qquad \gamma = g_2(\point{y}) \]Since all four of these elements are in \( \GEN \), the terms \(f_k, g_k\) must in fact be right-invertible. By Proposition \ref{prop:RinvPreservesPropto},
  \[ \gamma = g_2(\point{y}) \propto g_2(\point{x}) = g_2 \circ f_1(\chi_v) \]Note that \(g_2 \circ f_1 \) is right-invertible.

  Case 1: \( \gamma = \chi_w\) for some \(w \neq v\).

  Without loss of generality, \(w\) is a red vertex. We have \( \chi_v^w = \chi_v^\clubsuit \), so
  \[ a_0 = \gamma^\clubsuit \propto g_2 \circ f_1 (\chi_v^\clubsuit) = g_2 \circ f_1 (\chi_v^w) \propto \gamma^w = a_1 \]which is impossible.

  Case 2: \( \gamma = \chi_{e,\clubsuit} \) for some edge. Then \(e\) contains an endpoint \(w \neq v\), which we may suppose again to be red.

  Since \( \chi_v^w = \chi_v^\clubsuit \),
  \[ 2 = \gamma^w \propto g_2 \circ f_1 (\chi_v^w) = g_2 \circ f_1(\chi_v^\clubsuit) \propto \gamma^\clubsuit = z \]But this is likewise impossible.\pfend{\ref{claim:EDGEGEN defines edge gens}}
  
  With this in hand, we know that the set \( \VERTEXGEN \) of all \( \point{x} \in \GEN \) which are not in \( \EDGEGEN \) is (uniformly first-order) definable. This set is, of course, better known as the set of all \( \point{x} \) which are \( \sim \)-equivalent to one of the \( \chi_v \).

  Lastly, let \( \EDGE(x,y) \) be a formula asserting that
  \begin{quotation}
    \( x \in \VERTEXGEN \) and \( y \in \VERTEXGEN \) and there exist \( x' \sim x \), \( y' \sim y \) and \( w \in \EDGEGEN \) such that \( w \propto x' * y' \).
  \end{quotation}

  \begin{claim}\label{claim:EDGE defines edges}
    For \( \point{x}, \point{y} \in \VERTEXGEN \), \( \alg{D} \models \EDGE(\point{x}, \point{y}) \) iff there exists an edge \( e = \{v,w\} \) such that \( \point{x} \sim \chi_v \) and \( \point{y} \sim \chi_w \). 
  \end{claim}

  Proof: (\( \Leftarrow\)): If the red vertex \(v\) has an edge to the blue vertex \(w\), then
  \begin{align*}
    \chi_v * \chi_w &= \left( {a_1}_{|v} \oplus {a_0}_{|\mathrm{else}} \right) * \left( {b_1}_{|w} \oplus {b_0}_{|\mathrm{else}} \right) \\
    &= {a_1 * b_0}_{|v} \oplus {a_0 * b_1}_{|w} \oplus {a_0 * b_0}_{|\mathrm{else}} \\
    &= 2_{|v} \oplus 1_{|w} \oplus 0_{|\mathrm{else}} \\
    &\propto 2_{|v} \oplus 1_{|w} \oplus z_{|\clubsuit} \oplus 0_{|\mathrm{else}} \\
    &= \chi_{e, \clubsuit}
  \end{align*}

  (\(\Rightarrow\)): Assume \( \EDGE(\point{x},\point{y}) \). Fix 
  \begin{align*} \point{x}' &\sim \point{x} \sim \chi_v \\  
    \point{y}' &\sim \point{y} \sim \chi_w \\ 
    \point{x}' * \point{y}' &\propto \point{w} \sim \chi_{e,\clubsuit} \end{align*}(The proof is the same if \(w \sim \chi_{e,\spadesuit}\).)

  Since all these points are members of \( \GEN \), we may choose right-invertible terms so that
  \[ \chi_{e, \clubsuit} = f(\point{w}) \qquad \point{x}' = g(\chi_v) \qquad \point{y}' = h(\chi_w) \]
  Then \( f \) is right-invertible and 
  \[f(\point{w}) \propto f(\point{x}' * \point{y}') = f(g(\chi_v) * h(\chi_w)) \]We will be done if we can show that \(e = \inangle{v,w}\).

  If this were false, we could choose an endpoint \(u \in e \setminus \{v,w\}\), which we may suppose is red. Then
  \[ \chi_v^u = \chi_v^\clubsuit = a_0 \qquad \chi_w^u = \chi_w^\clubsuit = b_0 \]so
  \begin{align*} 2 = \chi_{e,\clubsuit}^u = f(w^u) \propto f(g(\chi_v^u) * h(\chi_w^u))& \\= f(g(\chi_v^\clubsuit) * h(\chi_w^\clubsuit)) & \propto f(w^\clubsuit) = \chi_{e, \clubsuit}^\clubsuit = z \end{align*}a contradiction.
  \pfend{\ref{claim:EDGE defines edges}}

  Observe that since \( \mathbb{G} \) has no isolated vertices, the subsets \( \VERTEXRED \) and \( \VERTEXBLUE \) of \( \VERTEXGEN \) consisting of those \( \point{x} \) which are \( \sim \)-equivalent to a red (resp.~a blue) vertex are definable using the \( \EDGE \) relation.
  
  The foregoing shows that
  \[ \inangle{ \VERTEXRED/\sim, \VERTEXBLUE/\sim \; ;\; \EDGE } \]is isomorphic to our original bipartite graph \( \mathbb{G} \); since all the relations in this isomorphism are uniformly definable, we have effected a semantic embedding of bipartite graphs into \( \HSP{\frzflt{\alg{A}}{\tau}} \).
\end{proof}

\begin{proof}[Proof of Theorem \ref{thm:main}]
  Theorem \ref{thm:bipartite into frzflt} shows that, if \( \alg{A} \) has a \(\tau\)-boxmap depending on too many variables, then \( \HSP{\frzflt{\alg{A}}{\tau}} \) is hereditarily finitely undecidable. But we have already seen in Lemma \ref{lemma:frzflt into A} that \( \HSP{\frzflt{\alg{A}}{\tau}} \) semantically embeds into \( \HSP{\alg{A}} \). Since semantic interpretability is transitive, we are done.
\end{proof}

\section{Problems}

We have seen that the definition of \( \frzflt{\alg{A}}{\tau} \) only makes sense when \( \tau \) is strongly abelian. The reason for introducing the intermediate language \( \frz{L}{\tau} \) is, we hope, to allow us to get a better handle on congruence intervals admitting both types 1 and 2 in these varieties. It is known (\cite{Val1994}) that the (1,2) and (2,1) transfer principles must hold in solvable congruence intervals. This can be read to say that, if we consider the congruence lattice of an algebra lying in a finitely decidable variety and consider the greatest congruence \( \sigma_1 \) such that \( [\bot,\sigma_1]\) has only type-1 covers, and likewise \( \sigma_2\), then these two congruences act a bit like direct factor congruences. However, the following problems are open:

\begin{problem}
  In a finite algebra \( \alg{A} \) in a finitely decidable variety, must \( \sigma_1 \) permute with \( \sigma_2 \)?
\end{problem}

\begin{problem}
  If \( \alg{A} \) is abelian, must \( \alg{A} \cong \alg{A}/\sigma_1 \times \alg{A}/\sigma_2\)? 
\end{problem}

\begin{problem}\label{problem:xp}
  If \( \tau = \sigma_1 \lor \sigma_2\) denotes the solvable radical of \( \alg{A} \), must \( \frz{\alg{A}}{\tau} \) be the direct product of \( \frz{\left( \alg{A}/\sigma_1 \right)}{\tau} \) and \( \frz{\left( \alg{A}/\sigma_2 \right)}{\tau} \) in the sense of \( \frz{L}{\tau} \)?
\end{problem}

\begin{problem}
  The same as \ref{problem:xp}, except with the added assumption that every congruence of \( \alg{A} \) is comparable to \( \tau \).
\end{problem}

Lastly, our arguments in this paper have used that whenever \( \tau \) is strongly abelian, the variety \( \HSP{\frzflt{\alg{A}}{\tau}} \) is finitely axiomatizable. We suspect that finite axiomatizability should hold much more broadly:

\begin{problem}
  \begin{enumerate}
  \item If \( \var{V} \) is a finitely decidable locally finite variety, is \( \var{V} \) finitely axiomatizable?
  \item Same question, but restricted to finitely generated \( \var{V} \).
  \item If the finite algebra \( \alg{A} \) with solvable radical \( \tau \) generates a finitely decidable variety, is \( \HSP{\frz{\alg{A}}{\tau}} \) finitely axiomatizable?
  \end{enumerate}
\end{problem}

\end{document}